\documentclass[12pt]{article}
\usepackage{amsmath,amssymb,amsthm}
\usepackage{pgf}

\usepackage{graphicx}
\usepackage[figuresright]{rotating}

\newtheorem{theorem}{Theorem}[section]

\newtheorem{definition}[theorem]{Definition}

\newtheorem{corollary}[theorem]{Corollary}
\newcommand{\Z}{\mathbb Z}
\newcommand{\R}{\mathbb R}
\newcommand{\N}{\mathbb N}

\pgfdeclareimage[height=120pt]{poly}{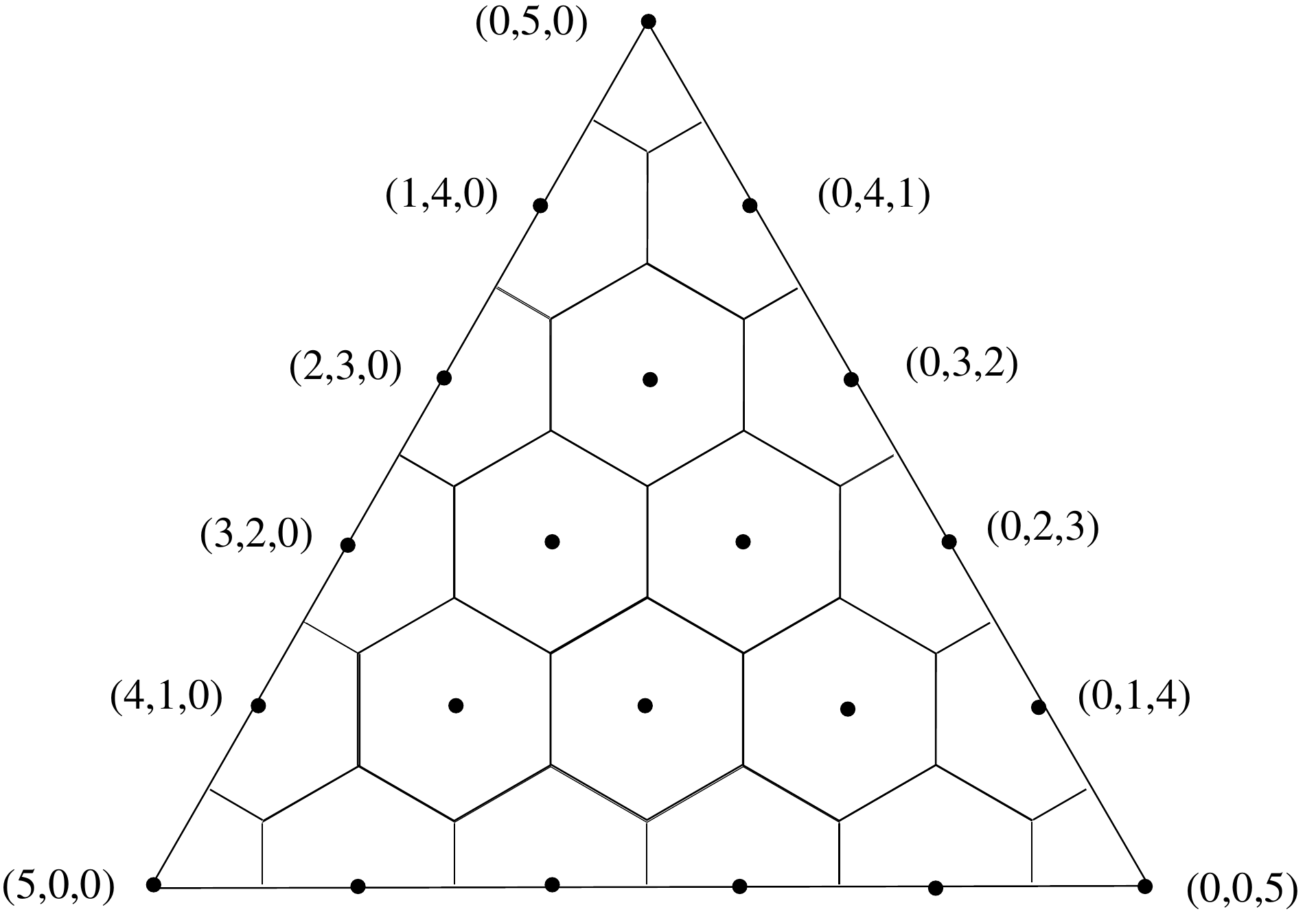}

\title{Apportionment Methods}
\author{Horst F. Niemeyer and Alice C. Niemeyer$^1$\\
$^1$ School of Mathematics and Statistics\\The University of\\
  Western Australia\\35 Stirling Highway\\Nedlands, WA 6009\\Australia} 

\begin{document}

% typeset front matter
\maketitle

\begin{abstract}
Most democratic  countries use election methods  to transform election
results into whole numbers which usually give the number of seats in a
legislative body the parties  obtained.  Which election method does this
best can  be specified  by measuring the  error between  the allocated
result  and the  ideal proportion.  We show  how to  find  an election
method  which is  best  suited to  a  given error  function.  We  also
discuss several properties of election methods that have been labelled
paradoxa.  In particular we  explain the highly publicised ``Alabama''
Paradox for  the Hare/Hamilton  method and show  that other  popular election
methods come with their very own paradoxa.
\end{abstract}

\section{Introduction}\label{sec:intro}

Practically  all democratic countries  are faced  with the  problem of
selecting members  of their legislative  bodies according to  votes of
their   population.   The  method   by   which   this  selection   of
representatives is  performed is  commonly known as  an \emph{election
method}. Its main function is  to transform the election results, which
are usually  the number  of votes for  various candidates  or parties,
into whole  numbers which usually give the number  of seats in a
legislative  body. Nearly  every  democratic country  employs its  own
favourite election method. This immediately leads to the question of  which
election method is in some sense optimal and most just? This is not an
easy  question to  answer  as one  can  see from  the  vast amount  of
literature  it has generated.  History  shows  that  at various  times
different  election methods  were in  favour and  many  countries have
changed their election method in the past.

One  can distinguish  between two  major types  of  election methods,
namely the  \emph{majority voting methods}  and the \emph{proportional
voting methods}.  The majority voting  method is used, for example, in
Great Britain to elect members  of parliament and in the United States
of  America to  elect members  of the  congress. In  this  method, the
eligible voters are  divided into voting districts and,  in its purest
form, each district elects one  candidate by majority vote. It is well
known  that  the  percentage  of  members in  parliament  or  congress
belonging to a given party need not be close to the percentage of votes
this party obtained overall. For  this reason, we do not consider this
voting method in this paper.

Among the  most frequently  discussed proportional voting  systems are
the  D'Hondt or  Jefferson method  (e.g. used  in Germany  until 1983,
still  used in  many countries),  the  Hare or  Hamilton method  (used
e.g.\  in  Germany,  Tasmania,  used  in  the US  1840  to  1890)  and
Sainte-Lagu\"e or Webster method (used in New Zealand).

The main  aim of this paper is  to present a general  treatment of all
proportional voting methods.  Mathematically, the accuracy of a voting
method  can  be  measured  by  a so-called  \emph{error  function},  a
function which  measures the  gravity of the  error between  the exact
votes  and  the  allocated  seats.  In practice  error  functions  are
dictated  by courts or  legislative bodies  and these  error functions
need  not coincide  with  mathematical error  functions.   We show  in
Theorem~\ref{the:main} how to find  a proportional voting method which
minimises a given  error function.  A slightly weaker  version of this
theorem   has   been   announced   by  Niemeyer   and   Wolf   (1984),
\cite{NiemeyerWolf84}.   Here  we give  a  first  complete proof.   We
demonstrate  how  most of  the  well  known  election methods  can  be
obtained  from  this  general  result  by  special  choices  of  error
functions.  The  Hare method can  be singled out  by the fact  that it
minimises  infinitely   many  different  error   functions.   This  is
certainly   not   the   case   for   the  methods   of   D'Hondt   and
Sainte-Lagu\"e. Another  aim of this  paper is to give  a mathematical
argument as to why the Hare  method is the best among the proportional
voting methods.

Each voting method  can lead to results which  appear paradoxical.  We
discuss various paradoxa which have  enjoyed a lot of attention in the
literature, for  example, the so-called  Alabama Paradox for  the Hare
method.  We argue  that in many cases the "paradoxity"  is only in the
eye  of the  beholder.   In  particular, the  methods  of D'Hondt  and
Sainte-Lagu\"e are blessed with their own paradoxa.

\section{Background}\label{sec:back}

This paper is concerned  with four different but mathematically
related problems. First there  is the apportionment problem as seen
in the example of the  American House  of Representatives of  the
United  States Congress, where each state of the union is  entitled
to a number of seats, at least one, according to and as closely
proportional as possible to  its number  of legal inhabitants.
The latter  is determined every ten years by a census of the
United States.  Each state is subdivided into voting districts
according to how many seats are assigned to it and the representatives
of each district are assigned by a majority vote.

Second  there  is  the  apportionment problem  for  political  parties
participating  in  an  election  of  a  country  where  seats  in  the
parliament are  assigned to parties  as proportionally as  possible to
the election  results.  As an example  we can take  the German Federal
election.  Germany  is divided into  voting districts.  The  number of
seats in parliament  is twice the number of  voting districts.  Voters
receive  a ballot  on  which there  are  two votes.   The second  vote
(Zweitstimme) is cast  for a political party.  The  number of seats in
parliament  is  assigned  to  eligible parties  as  proportionally  as
possible to these  second votes.  (A party is  eligible if it received
at least  5\% of the  total number of  second votes or at  least three
seats by first vote.)  The  first vote (Erststimme) is to determine by
simple  majority   a  representative   for  the  voting   district  in
parliament.  The elected candidates  receive seats in parliament which
are counted  towards the allocated number  of seats of  the party they
represent.    The   remaining   seats   are  distributed   among   the
participating parties  according to the second votes  and according to
the election  method used.   It may happen  that a party  obtains more
seats by the first vote  (Erststimme) than is allowed according to the
second vote (Zweitstimme).  The parties  keep those seats and they are
called  ``\"Uberhangmandate''.  Using a  first  and  second vote,  the
Federal  elections  combines the  important  feature  of the  majority
election methods, that each voting district is represented by a member
of parliament, with the advantage of proportional election methods, by
which  the  number  of  seats  allocated  to a  party  is  as  closely
proportional to the number of  votes as possible.  Of course, there is
always the difficulty that one or more parties get more seats directly
than by the number of total votes.

The third problem is that of rounding a list of given numbers, so that
the  total sum  of the  rounded numbers  equals the  rounded (standard
rounded)  sum.   The  easiest  example  perhaps is  that  of  rounding
percentage  numbers,  which add  up  to  100\%,  but after  (standard)
rounding fail to  add up to 100\%.  Another problem  can be the amount
of exports of a certain commodity given in Euro by each country of the
European  community,  for   instance.   The  press  publishes  figures
representing the exports of each country rounded to millions of Euros.
The total amount is also  rounded to the nearest million independently
of  the individual  figures.   Now  the problem  is  that the  rounded
figures should add up to the rounded total.

The fourth problem comes from  Operations Research. Let us assume that
a  factory is  producing indivisible  goods,  e.g.\ cars  of the  same
make.  The profit  is to  be distributed  among the  share  holders as
closely proportional to the number of shares as possible.

\subsection{Describing the four problems mathematically}

In all  of these problems  we can start  with a positive  integer $M$,
where  $M$ represents  the number  of seats  in the  first  and second
problem;  $M$ denotes the  sum of  the rounded  integers in  the third
problem;  $M$  is   the  number  of  goods  produced   in  the  fourth
problem. Let $n$ denote the number of states in the Union, or parties,
or numbers to be rounded,  or shareholders in a manufacturing company.
Further,  we   define  a  real  $n$-dimensional  vector   ${\bf  a}  =
(a_1,\ldots,  a_n)$ and  an integer  vector ${\bf  m} =  (m_1, \ldots,
m_n)$. Let  $A$ denote  the sum of  the entries  of ${\bf a}.$  In the
first problem  $a_j$ is the number  of votes in the  $j$-th state, $A$
the total number  of votes and $m_j$ is the  number of seats allocated
to the  state. Let $q_j =  a_jM/A$ be the \emph{exact  quota}, that is
the  exact proportion  of seats  the $j$-th  state may  claim.  In the
second  problem $a_j$ is  the number  of votes  party $j$  obtains and
$m_j$  is the  number of  seats  allocated to  party $j$.  Let $q_j  =
a_jM/A$ be  the \emph{exact  quota}, that is  the exact  proportion of
seats the $j$-th  party may claim.  In the third  problem $a_j$ is the
$j$-th given number and $m_j$  is the integer obtained by the rounding
process  applied  to $a_jM$.  Here  let $q_j  =  a_j$.  In the  fourth
problem, $a_j$ is the exact quota  to which the $j$-th share holder is
entitled to  and $m_j$ the number  of goods he receives.  Again we let
$q_j = a_j$.

This leads to the following  definitions. We are interested in sets of
vectors whose  entries are  the exact quotas.   Let $M$ be  a positive
integer  (e.g.  number  of  seats)  and $A$  a  positive  real  number
(e.g. number of valid votes).  Let $${\mathcal Q} =\{ (q_1,\ldots,q_n)
\mid q_j  \in \R,\, q_j \ge 0,\,  \sum_{i=1}^n q_i = M  \}.$$ This set
contains  all  possible  exact  quotas  among $n$  parties  (where  in
problems one and two we computed  the exact quotas from a total of $A$
valid  votes). Note  that  we  allow ${\mathcal  Q}$  to contain  real
vectors.  Let $${\mathcal M} =  \{ (m_1,\ldots,m_n) \mid m_j \in \Z,\,
m_j \ge 0  \mbox{\ and\ } \sum_{j=1}^n  m_j = M \}$$ be  the subset of
all  integer  vectors  whose  entries  are  non-negative  and  sum  to
$M$.  Note that  ${\mathcal M}$  is a  lattice over  the  integers and
represents the possible seat allocations.

\begin{center}
\begin{figure}

\pgfputat{\pgfxy(4.8,-2)}{\pgfbox[left,center]{\pgfuseimage{poly}}}
\vspace*{3.5cm}
\begin{center}\caption{Election Polyhedra}\end{center}
\end{figure}
\end{center}

We illustrate  the set ${\mathcal Q}$  for 3 parties and  for $M=5$ in
Figure~1. The set ${\mathcal Q}$ is a unilateral triangle in the space
$\R^3$.  Its corners are $(5,0,0)$, $(0,5,0)$ and $(0,0,5)$. The black
dots  represent  the  exact   quotas  which  are  also  possible  seat
distributions, that is exact quotas  which lie in ${\mathcal M}$.  The
other points  in the triangle  correspond to all points  in ${\mathcal
  Q}.$  A valid  apportionment method  should map  every point  in the
triangle to a seat allocation.  It should be possible to draw a region
around each black dot containing only this one black dot such that all
exact quotas inside  this region are mapped to  the seat allocation to
which the black dot is mapped.  We believe a good apportionment should
map a black dot to the seat allocation it describes.

Finally, the apportionment method might have a choice 
of which seat allocation to choose for the exact quotas which lie
on the borders between two or more possible regions.

An apportionment method is in principal a function which maps each
possible vote distribution to a certain seat distribution. For a
precise definition see p.~97 Balinski and Young (1982).

We believe a good apportionment method should satisfy
\begin{equation}\label{eq:good}
\mbox{If\ } {\bf  q}\in {\mathcal M} \mbox{\ then\ } f({\bf q}) =  \{{\bf q}\}.
\end{equation}

In most cases ${\bf a}$ and $M$ are fixed. Sometimes it might
be necessary to emphasise the dependence of $f$ on ${\bf a}$ and $M$, 
in which case $f_{{\bf a},M}$ denotes an apportionment method with given
${\bf a}$ and $M$.

While a mathematical definition of an apportionment method is
sufficient to 
procure a meaningful apportionment, considerations of fairness dictate 
additional requirements. Especially the order in which the parties
are listed on a ballot should have no effect on the outcome of the
election, that is the apportionment method should be symmetric. Formally,
let $\pi$ be a 
 permutation of $\{1, \ldots, n\}$  and let ${\bf q} \in
 {\mathcal Q}$ be the
 $n$-dimensional real vector $(q_1,\ldots, q_n)$. If we define 
 ${{\bf q}^\pi} =  (q_{1^\pi}, \ldots, q_{n^\pi})$.  If $B$ is a
 subset of ${\mathcal Q}$ then define $B^\pi = \{ b^\pi \mid b \in B \}.$

\begin{definition}$($Symmetry of an apportioning method$)$ 
An  apportionment method is \emph{symmetric}, if for every permutation $\pi$
 of $\{1,\ldots,n\}$ we
 have  $f({\bf q}^\pi) = \left(f({\bf q})\right)^\pi$. 
\end{definition} 

The outcome of a  reasonable election method  should depend only on
the vectors of the exact quotas, that is if two different vote distributions
yield the same exact quotas, the results should be the same.
Thus the apportionment method
should be homogeneous.

\begin{definition} $($Homogeneity$)$ 
An  apportionment method is 
\emph{homogeneous} 
if for every real $\lambda >0$
we have $f_{\lambda {\bf a},M}( {\bf q}) = f_{{\bf a},M}({\bf q}).$
\end{definition}

From now on all apportionment methods are symmetric and homogeneous.

%proportionalitaet

%election simplex

%divisoren Verfahren is exact quota ok

%We would like to comment on the so-called 5\%-clause, which excludes
%all parties in a federal election that obtained less than 5\% of the
%votes or obtained at most one seat via the first vote (`Erststimme').

\section{Overview over various apportionment methods}

Historically, many apportionment methods have been discussed in
political assemblies and in the literature. In this paper we
are concerned with two major classes:  the divisor methods
and the rounding methods. Examples of the divisor
methods are the methods of D'Hondt and Sainte-Lagu\"e. The best
known rounding method is the Hare method, also called the method of
the greatest remainders. A detailed description of these methods can be
found in Balinski and Young (1982), \cite{BalinskiYoung}. Here we give
a very brief overview in 
order to establish the notation we use for the various methods. 

The \emph{divisor methods} start with a strictly increasing sequence
$(d_j)_{j \in \N}$ of
non-negative real numbers, called the \emph{sequence of divisors}. We
then have to divide the exact quotas  
$q_k$ for $k=1,\ldots, n$ by the $d_j$ and define the $M\times n$ matrix
whose entry in row $j$ and column $k$ is 
$q_k/d_j.$
We select the $M$ largest entries of this matrix and count the number
$m_k$ of entries selected from the $k$-th column.  Then the party $P_k$
will get $m_k$ seats. A \emph{linear divisor method} is given by an
arithmetically increasing sequence 
$(d_j)_{j \in \N}$, that is $d_j =  d(d_0 + (j-1)),$
where $d$ and $d_0$ are fixed real numbers. As $d$ divides all terms
of the sequence $d_j$, we can omit $d$ by the homogeneity of the
apportionment method, that is  we can choose as sequence of divisors
the sequence $d_j = d_0+ (j-1).$ 

Linear divisor methods have been treated extensively in the
literature, see for example the literature review
by Heinrich et al. (2005), \cite{Heinrichetal05}.
The following sequences yield some well known linear divisor
methods (see \cite[p. 124]{Kop}):

%\begin{table}[htb]
%\caption{Some linear divisor methods}
\begin{center}
\begin{tabular}{ll}
$d_0$ & Method\\
\hline
0 & Adams \\
1/3 & Danish Method\\
2/5 & Condorcet\\
1/2 & Sainte-Lagu\"e or Webster\\
2/3 & Considerant\\
1 & \mbox{D'Hondt or Jefferson}\\
2 & Imperiali\\
\end{tabular}
\end{center}
%\end{table}

Note that for $d_0 \le 1$ the linear divisor methods satisfy
Equation~(\ref{eq:good}).

\emph{Non-linear divisor methods} have also been considered. The most 
well-known ones are the Dean method, where $d_j = j(j-1)/(j-1/2)$,
and the method  of Hill or Huntington, where $d_j = \sqrt{j(j-1)}.$

As we will see later, the values $d_0 = 0$ is biased towards towards
small parties, $d_0 = 1$ is biased towards large parties, whereas the
value $d_0 = 1/2$ is considered to be neutral.

Let  $\rho$ be real constant with $0 \le \rho \le 1$.
The \emph{$\rho$-rounding method} (see
Kopfermann~\cite[p.\ 117]{Kop}) considers the exact quotas $q_j$
and put 
$q_j^\rho :=  \frac{q_j}{M} (M + 2\rho - 1).$ Then we have that
$\sum_{j=1}^n q_j^\rho =   (M + 2\rho -1)$ and
let $\mu_j^\rho  = \lfloor q_j^\rho \rfloor$. Then
\begin{equation}\label{eq:mm}
\sum_{j=1}^n \mu_j^\rho \le \left\{
\begin{array}{ll} M-1 & \mbox{for\  } \rho = 0 \\
M & \mbox{for\  } 0 < \rho < 1 \\
M+1 & \mbox{for\  } \rho = 1 \\
\end{array}
\right.
\end{equation}

For $0 \le  \rho < 1$ we have therefore  $\sum_{j=1}^n q_j^\rho \le M$
and  we  define $r_j^\rho  =  q_j^\rho  -  \mu_j^\rho$.  We  start  by
allocating  $\mu_j^\rho$  seats  to  party  $P_j$. This  leaves  $M  -
\sum_{j=1}^n  \mu_j^\rho$ or,  when $\rho  =  0 $,  $M -  \sum_{j=1}^n
\mu_j^\rho +1$,  unallocated seats. Party $P_j$  obtains an additional
seat if $r_j^\rho$ is among the $M - \sum_{j=1}^n \mu_j^\rho$ (or $M
- \sum_{j=1}^n \mu_j^\rho + 1$)  largest numbers of $r_1^\rho, \ldots,
r_n^\rho.$  If $\rho=0$  there is  still a  problem if  all  $q_i$ are
integers.  In  this case  all $r_j^0 =  0$ and  we have to  assign one
additional seat.  This is allocated to a party at random.  In the case
$\rho = 1$  and all $q_i$ are integers, we have  assigned one seat too
many and  this seat is taken  from a random party  among the parties
who obtained more seats than their exact quota.

For $\rho=1/2$  we obviously have the well-known method of the greatest
remainder. This method is also known under the name Hare or Hamilton.
 For all other values of $\rho$  ($0 \le \rho \le 1$) we have a
general rounding method or a general greatest remainder method. We will
discuss these methods especially with respect to their effect on the
apportionment method a little later.

Note that for $0 < \rho < 1$ the $\rho$-rounding methods satisfy
Equation~(\ref{eq:good}).

\section{Constraints}\label{sec:constraints}

In this section we examine  the constraints of  apportionment methods,
which are  necessary or  beneficial for solving the four problems
considered above.

There are several possibilities to enforce
certain conditions. Not all of these conditions can be fulfilled
exactly
but it is also important to be able to compute the probability with which
a condition is violated when choosing an apportionment method.

These conditions are as follows:

\begin{enumerate}
\item {\sc Bias Condition:}\quad
The apportionment method should be free of bias, that is it should
neither favour large 
 nor small parties. Suppose $L, S$ are subsets of $\{1,\ldots, n\}$
 such that $m_j > m_i$ whenever $j \in L$ and $i \in S.$ Then an
 apportionment method
 favours large parties if 
$$\frac{\sum_{i\in L}m_i}{\sum_{i \in L} a_i} > 
\frac{\sum_{j\in S}m_j}{\sum_{j \in S} a_j}$$   and it favours small
parties if 
$$\frac{\sum_{i\in L}m_i}{\sum_{i \in L} a_i} < 
\frac{\sum_{j\in S}m_j}{\sum_{j \in S} a_j},$$  see \cite[p.\
  125]{BalinskiYoung}. 
\item {\sc Monotony Condition:}\quad  If one party has a larger exact quota
  than another,
it cannot receive less seats, that is if $q_j < q_k$ then $m_j \le m_k$
for all $j, k$ with $j, k \in \{1,\ldots, n\}.$
\item {\sc Lower Quota Condition:}\quad Each party is assigned at least as many
seats as the largest integer less than or equal to the exact quota, that is
$\lfloor q_j \rfloor \le m_j$ for all $j$ with $1\le j\le n.$
\item {\sc Upper Quota Condition:}\quad Each party receives at most as many
seats as the least integer greater or equal to the exact quota, that is
$m_j \le \lceil q_j \rceil$ for all $j$ with $1\le j\le n.$
\item {\sc Majority Condition:}\quad If a party obtains the absolute majority
of the votes then it receives the absolute majority of the seats, that is if
for some $j$ with $1\le
j\le n$ we have $a_j > \frac{1}{2} A$  (and hence $q_j >
\frac{1}{2}M$) then $m_j > \frac{1}{2} M$.

\item {\sc Coalition Condition:}\quad If a party has less than half
  of the total
number of votes then it receives also less than half of the total number
of seats, that is if 
 for some $j$ with $1 \le j \le n$ we have
$a_j < \frac{1}{2} A$ (and hence $q_j < \frac{1}{2}M$)  then $m_j <
  \frac{1}{2} M$.   

This conditions is called Coalition Condition
as it ensures that in a three party method the coalition of the  two
smaller parties which received together the absolute majority of the votes,
will also receive the absolute majority of the seats.
\item {\sc Independence Condition:}\quad The number $m_j$ of seats assigned to
the party $P_j$ depends only on the exact quota $q_j$ but not on the
distribution of the quotas $q_k$ of other parties $P_k$ for $k \neq j.$
\item {\sc House Monotony:}\quad
Let ${\bf a}$ be a vector of votes.
Let ${\bf m}$ be the seat distribution
obtained from ${\bf a}$ in
the case that there are $M$ seats and 
$\tilde{\bf m}$ the seat distribution obtained from ${\bf a}$ in
the case that there are $M+1$ seats. Then a House monotone 
 apportionment method satisfies
${\bf m}  \le \tilde{\bf m}.$
\end{enumerate}

The following table indicates when 
the conditions are satisfied 
for the $\rho$-rounding methods and the linear divisor methods for 
given $d_0$.

\bigskip
\noindent
\begin{tabular}{l|ll}
Condition & $\rho$-rounding method & $d_0$ linear divisor method\\
\hline
Homogeneity & always  & always \\
Unbiased & $\rho = 1/2$ & $\delta_0 = 1/2$ 
\cite[Prop.~5.3]{BalinskiYoung}\\  
Monotony & always  & always \\
Lower Quota & $1/2 \le \rho \le 1$ \cite[p.120]{Kop} &
$\frac{n-2}{n-1} \le d_0 \le 1$
\cite[p.131]{Kop}\\  
Upper Quota & $0 \le \rho \le 1/2$ \cite[p.120]{Kop} &
 $0\le d_0\le \frac{1}{n-1}$ 
\cite[p.131]{Kop}\\  
Majority & $\rho=1,\,$ $M$ odd \cite[p.121]{Kop}
& $d_0 = 1,\,$ $M$
odd \cite[p.131]{Kop}\\
Coalition & $\rho=0,\,$ $M$ odd \cite[p.121]{Kop}
& $d_0 = 0,\,$ $M$ odd
\cite[p.131]{Kop}\\
Independence & $n\le 2$ or  $m_i=m_j$ \cite[p.97]{Kop}& 
$n\le 2$ or  $m_i=m_j$ \cite[p.97]{Kop}\\
House Monot. & never & all 
\cite[Cor.~4.3.1, p.117]{BalinskiYoung}\\  
\end{tabular}

\bigskip
All non-linear divisor methods are homogeneous. None is unbiased, see
\cite[Prop.~5.3]{BalinskiYoung}, all are monotone. 
No non-linear divisor method satisfies both the upper and lower quota
condition, however it never violates both simultaneously, see
\cite[Prop.~6.4 and 6.5]{BalinskiYoung}.  All non-linear divisor
methods are House monotone, see
\cite[Cor.~4.3.1, p. 117]{BalinskiYoung}.

A method for computing the seat bias for a given apportionment method
with a hurdle (e.g. the 5\% hurdle) can be  found in 
Schwingenschl\"ogl and Pukelsheim (2006), \cite{SchwingenPukel06}.
Another condition, a Gentle Majority Condition,
which is important for forming committees is
discussed in Pukelsheim (2006), \cite{Pukelsheim06}.
 
\section{A modification of Hare and the
history of the Hare-Niemeyer Method in Germany}
It  all  started  with   an  article  in  the  newspaper  ``Frankfurter
Allgemeine" (FAZ)  by Dr. K.F.~Fromme which  appeared on  14 October
1970, pointing out the difficulties in determining the number of seats
each  party  gets in  the  various  committees  pursuant to  the  1970
elections in the  Federal Republic of Germany.  In  this election, the
CDU/CSU  won 253  seats in  parliament, the  SPD 237  and the  FDP 28,
giving  the  SPD/FDP  coalition  under  Chancellor  Helmut  Schmidt  a
majority of 265 seats (including  the members from Berlin, who did not
always have a vote but were  counted in the assignment of seats in the
committees). The  D'Hondt system was used  at that time  to determine the
number of parliamentary  seats a given party won  in general elections
and to determine the distribution of committee seats.

The difficulty  that Fromme  pointed out was  the paradoxon  that - in a
committee with  a given number of 33  seats - the distribution according
to D'Hondt was  as follows: the CDU/CSU was assigned  17 seats, SPD 15
and the FDP  l, thus giving the opposition party  a majority. The same
is also true e.g. for  committees with $33,31,29,\ldots ,9$ members. This led
to  a  discussion  in  parliament  about  changing  the  size  of  the
committees because  it was assumed  that a method better  than D'Hondt
``would be  hard to find". The  political question which  now arose was
how to keep the majority and the mathematical question concerned which
method to use.   After giving the matter some  thought, on 16 October 
1970  the first author  wrote to  the administration of  the Bundestag
suggesting the   method of the largest  remainders, which was
subsequently adopted.   

Almost seven years later, the  CDU/FDP coalition in the state
Niedersachsen wanted  to introduce  legislation  which would
replace the D'Hondt-system  by the Hare-system.
But the procedure left a couple of questions open. 
Therefore, the committees  of the state parliament which were discussing
this piece of legislation invited the first author
to a  joint hearing on  23 March 1977.   The
problem they discussed was the following: Even with the Hare method
it can happen, that a coalition of parties with the
majority   of the seats in parliament   does not  get a
majority of the seats in a committee.

As an example to demonstrate that the Hare method does
not always satisfy the  Majority Condition consider three parties
competing for $M=101$ seats. They received $50600, 40650$ and $9750$
votes, respectively. The exact quotas
they receive are  $q_1  =  50.6; q_2   =  40,65, q_3 =  9.75.$   
Then, according to the Hare method, the seat distributions are 
$m_1 = 50, m_2 = 41$ and $m_3 = 10$ and so the first party does not
have the majority of the seats, even though it received the majority
of the votes.

In general, the  Hare  method does
not always  satisfy the Majority  Condition when $M$  is odd. If  $M =
2k+1$, then it is possible that  a party obtains more than 50\% of the
votes but only  receives $k$ seats.  However, it  is not possible that
the party receives less than $k$ seats as is possible using the
Sainte-Lagu\"e method. Further, it  is also fairly easy to fix this
situation as suggested by the first author.
If a party obtains more than  50\% of the votes
then  its quota  is at  least $k  +  1/2.$ If  one gives  one of  the
remaining seats to this party, then the Majority Condition is fulfilled
and it is still possible to redistribute the remaining seats such that
each party gets at least the lower quota. 
D'Hondt's method favors the  largest  party so much so, that
it can receive more than $q+1$ seats, where $q$ is the
exact quota. This is because this method
does not satisfy the Upper Quota Condition.

As a result  of  this  hearing,   the  parliament  in
Niedersachsen decided to hold the elections in Niedersachsen according
to the  modified Hare method, sometimes also  called the Hare-Niemeyer
procedure.  However, in 1986 this method was again replaced by the 
d'Hondt method (see \cite{niedersachsen}).

From 1987 onwards the Hare-Niemeyer method has also been used to for the seat 
allocations in the German Bundestag until it was replaced in 2008 by the
method of Sainte-Lagu\"e (see \cite{bundestag}).

\section{On the error function of a general apportionment
  method}\label{sec:error} 

We start with a general apportionment problem. Suppose $M$ seats in
parliament are to be distributed among
parties $P_1, \ldots, P_n$.
Suppose further the parties received votes
$a_1, \ldots, a_n$, represented by the vector
${\bf a} = (a_1, \ldots, a_n)$, and let $A = \sum_{i=1}^n a_i.$
Then the exact quotas are $q_j = \frac{M a_j}{A}$ for $1 \le j \le n$
and are exactly proportional to the number of votes  party $P_j$ received.
Let $f$ be an apportionment method.
In any apportionment method the transition from a vector ${\bf q}$
with real entries to a vector ${\bf m}$ with integer entries
is bound to cause  errors. On the one hand we can measure
individual errors between  $q_j$ and $m_j$ for
$1 \le j \le n$ and on the other hand there is a global error which
comprises the individual errors. The following definition captures the
properties of an error function which is decomposable into individual
error functions.

\begin{definition}
For $1 \le j \le n$ let $\varphi_j : \N \rightarrow \R_{\ge 0}$ be
functions such that the functions $H_j: \N^+ \rightarrow \R_{\ge 0}$
defined by   $H_j(\ell) = \varphi_j(\ell) -
  \varphi_j(\ell -1 )$ are increasing.
Then the  function $\psi : {\mathcal M} \rightarrow \R_{\ge 0}$ defined
by $\psi({\bf m}) = \sum_{j=1}^n \varphi_j(m_j)$
is called a \emph{decomposable error function}.
The function $\varphi_j(x)$ is called the \emph{$j$th
  component of $\psi$}.
\end{definition}

Further,  we have
\begin{equation}\label{eq:psi}
\psi({\bf m }) = \psi(0) + \sum_{j=1}^n \sum_{\ell=1}^{m_j} H_j(\ell) .
\end{equation}

Decomposable error functions are also considered by Gaffke and
Pukelsheim (2008).
For a given a decomposable error function $\psi$
the following theorem describes an algorithm called {\sc
  MinimalSolution} to determine the vector 
${\bf m_0} \in {\mathcal M}$ which minimizes the error function. Hence
for a given decomposable error function the theorem  can be used to
give rise to an apportionment method, namely the method which assigns
each vector ${\bf a}$ the vector ${\bf m}^0$ which minimizes the error
function, see also Niemeyer and Wolf (1984). 
For divisor methods see also  the Min-Max Theorem of Balinski and
Young (1982), \cite[p. 100]{BalinskiYoung}.

\begin{theorem}\label{the:main}
Let $\psi$ be a decomposable error function. Then
\begin{enumerate}
\item[$1.$] There is at least one solution ${\bf m_0} \in {\mathcal M}$ which
  minimizes the error function $\psi$;
\item[$2.$] ${\bf m}_0 = (m_1^0, \ldots, m_n^0)$ is a minimal solution if
  and only if $H_j(m_j^0 + 1 ) \ge H_k (m_k^0)$ for all
$j, k \in \{1, \ldots, n \}$;
\item[$3.$] ${\bf m}_0$ is unique if $H_j(m_j^0 + 1 ) > H_k (m_k^0)$ for all
$j, k \in \{1, \ldots, n \}$   with  $j \not=k$;
\item[$4.$] 
Consider the
  following algorithm {\sc MinimalSolution}: 
Let ${\mathcal  S}$ denote the set of ${\bf m} = (m_1, \ldots, m_n)$ such that
\begin{enumerate}
\item[$(a)$] $\sum_{j=1}^n m_j = M$,
\item[$(b)$] the multiset $\{ H_j(\ell) \mid 1 \le \ell \le m_j, 1\le j \le n \}$ contains
  $M$ smallest      elements of the matrix 
$$\left( \begin{array}{ccc}
H_1(1) & \ldots & H_n(1) \\
\vdots & & \vdots \\
H_1(M) & \ldots & H_n(M)
	 \end{array} \right).
$$ Then  ${\mathcal S}$ consists of all
minimal solutions. 
\end{enumerate}
\end{enumerate}

\end{theorem}

\begin{proof}\quad
\begin{itemize}
\item[1.] The existence follows from the fact that the image of ${\mathcal M}$
under $\psi$  is a
finite subset of $\R_{\ge 0}$ and therefore has a minimal subset. Hence
there is at least one element ${\bf m} \in {\mathcal M}$ for which
$\psi({\bf m}) $ is the minimum and so ${\bf m}$ is a minimal solution.
\item[2.]
Suppose first that ${\bf m}_0=(m_1^0, \ldots, m_n^0)$ is a minimal solution and
$j,  k \in \{1, \ldots, n\}$. If $j=k$  the result
follows since $H_j(\ell)$ is increasing.
Now suppose $j\not=k.$  Choose another solution
${\bf m}_1$  which differs from    ${\bf m}_0$
in exactly two components, namely $m^1_j = m^0_j + 1$  and
   $m^1_k = m^0_k -1$. Since ${\bf m}_0$   is minimal we have
   $\psi({\bf m}_0) \le \psi({\bf m}_1)$   and hence
\begin{eqnarray*}
\sum_{i=1}^n \varphi_i(m_i^0)  & \le &   \sum_{i=1}^n \varphi_i(m_i^1)
\end{eqnarray*} which implies
   \begin{eqnarray*}
\varphi_j(m_j^0) + \varphi_k(m_k^0) & \le &   \varphi_j(m_j^1)
 +  \varphi_k(m_k^1)\\
 & = &   \varphi_j(m_j^0 + 1)  +  \varphi_k(m_k^0 - 1)
\end{eqnarray*} and so
 \begin{eqnarray*}
\lefteqn{H_k(m_k^0) = \varphi_k(m_k^0) - \varphi_k(m_k^0-1)}\\
 & \le &
 \varphi_j(m_j^0+1)  -  \varphi_j(
m_j^0)  = H_j(m_j^0 + 1).
\end{eqnarray*}

On the other hand suppose  ${\bf m}_0$ is a solution for which
$ H_k(m_k^0)  \le H_j(m_j^0 + 1) $ for all $j, k \in \{1, \ldots, n\}$.
Let ${\bf m} = ( m_1, \ldots, m_n)$ be another solution.
Then
$\psi({\bf m}) = \sum _{j=1}^{n} \varphi_j (m_j)$. Compare the values
of $\psi$ at $m_j^0$ with those at arbitrary points $m_j$
and note
$$M =\sum_{j=1}^{n} m_j^0 = \sum_{j=1}^{n} m_j .$$ Therefore one can
divide the indices
$1,\ldots,n$, for which $m_j^0 \neq m_j$ into two disjoint sets
$J_1$ and $J_2$, according to whether  $m_j^0 < m_j$ or
$m_j^0 > m_j$. Then  $\sum_{j \in J_1}(m_j - m_j^0) = \sum_{j
\in J_2}(m_j^0-m_j)$.
This yields that
\begin{eqnarray*}
\lefteqn{ \psi({\bf m}) - \psi({\bf m}_0) }\\
  &= &
 \sum_{j\in J_1} \left( \varphi_j (m_j ) - \varphi_j( m_j^0) \right)
 +  \sum_{j\in J_2} \left( \varphi_j (m_j ) - \varphi_j( m_j^0) \right) \\
       & = &  L - R
 \end{eqnarray*}

and that
\begin{eqnarray*}
L & = & \sum_{j\in J_1} \left( \varphi_j (m_j )  -
\varphi_j( m_j^0) \right) \\
& = & \sum_{j\in J_1} \left(  H_j(m_j) + H_j(m_j - 1 )
+ \cdots + H_j( m_j^0 + 1 ) \right) \\
& \ge& \sum_{j\in J_1}\left( m_j - m_j^0 \right) \cdot H_j(  m_j^0 + 1 )\\
& \ge& \min_{j\in J_1}  H_j(  m_j^0 + 1 )\cdot \sum_{j\in J_1} \left(
m_j - m_j^0 \right).
\end{eqnarray*}

On the other hand

\begin{eqnarray*}
R & = & \sum_{j\in J_2} \left( \varphi_j (m_j^0 )  -
\varphi_j( m_j) \right) \\
& = & \sum_{j\in J_2} \left(  H_j(m_j^0) + H_j(m_j^0 - 1 )
+ \cdots + H_j( m_j + 1 ) \right) \\
& \le& \sum_{j\in J_2}\left( m_j^0 - m_j \right) \cdot H_j(  m_j^0 )\\
& \le& \max_{j\in J_2}  H_j(  m_j^0 ) \cdot  \sum_{j\in J_2}
\left( m_j^0 - m_j \right) \\
& \le& \max_{j\in J_2}  H_j(  m_j^0 ) \cdot  \sum_{j\in J_1}
\left( m_j - m_j^0 \right).
\end{eqnarray*}

Then finally
\begin{eqnarray*}
L - R & \ge & \min_{j\in J_1}
 H_j(  m_j^0 + 1 ) \cdot  \sum_{j\in J_1}
\left( m_j - m_j^0 \right) -
\max_{j\in J_2}  H_j(  m_j^0 ) \cdot  \sum_{j\in J_1}
\left( m_j - m_j^0 \right) \\
 & \ge & \left( \min_{j\in J_1}
 H_j(  m_j^0 + 1 )  -
\max_{j\in J_2}  H_j(  m_j^0 ) \right)
\cdot  \sum_{j\in J_1} \left( m_j - m_j^0 \right) \\
 & \ge & \left( \min_{j\in J_1}
 H_j(  m_j^0 + 1 )  - \max_{j\in J_2}H_j(m_j^0) \right).
\end{eqnarray*}
By or assumption, 
$ H_j(  m_j^0 + 1 )  \ge H_k(m_k^0) $ for all $j,k$ and so 
$ H_j(  m_j^0 + 1 )  - H_k(m_k^0) \ge 0$ for $j \in J_1$ and $k \in
J_2.$
This shows that $L - R \ge 0$ and therefore $\psi({\bf m} ) \ge \psi(
{\bf m^0}$ for any ${\bf m} \in {\mathcal M}.$ Therefore
${\bf m^0}$ is a minimal solution.

\item[3.]
If the condition $H_j(m_j^0 + 1 ) > H_k(m_k^0 )$  is also satisfied
for all $j$ and $k$ with
$j \neq k$
then $L - R > 0,$  that is, the minimal solution is unique.

\item[4.]  Let ${\bf m} \in {\mathcal S}$ with
${\bf m} = (m_1, \ldots, m_n)$. Suppose that
there are $j, k$ with $H_j(m_j + 1) < H_k(m_k)$. Then  
the multiset $ \{ H_i(\ell) \mid 1 \le \ell \le m_i, 1 \le i\le n\}$
does not contain $M$ smallest elements as we can replace
$H_k(m_k)$ by the smaller $H_j(m_j + 1)$. Thus for all $j,k$ we have
$H_j(m_j + 1) \ge H_k(m_k)$ and 
by (2) ${\bf m}$ is a
minimal solution.
\end{itemize}
\end{proof}

Sainte-Lagu\"e \cite[Satz 6.1.8]{Kop} knew that the 
Sainte-Lagu\"e method minimised the error function with
 $\varphi_j(x) = \frac{1}{q_j} (x - q_j)^2$. Balinski and Ramirez
\cite{BalinskiRamirez99} show that the linear divisor methods
minimize the error function with
 $\varphi_j(x) = \frac{1}{q_j} (x - q_j +d_0 - \frac{1}{2})^2$.

\begin{corollary}\label{cor:lindiv} 
The linear divisor method given by $d_0$ can
 be obtained from the algorithm {\sc  MinimalSolution} for the 
  decomposable error function $\psi$, with $j$-th component
 $\varphi_j(x) = \frac{1}{q_j} (x - q_j + d_0 -
  \frac{1}{2})^2$.
\end{corollary}

\begin{proof}
Note first that $\varphi_j(x)$
is a convex function. The corresponding
  decomposable error function $\psi$ always has a minimal solution.
Algorithm {\sc MinimalSolution} leads to the linear divisor method
described
by $d_0$, because we have that
$H_j(\ell) = \varphi_j(\ell) - \varphi_j(\ell-1)
=  \frac{2}{q_j}( \ell - q_j + d_0 - 1 )
= \frac{2(\ell +d_0 -1 )}{q_j} -2$ is strictly increasing.
 Consider the function
 $K_j(\ell) = \frac{q_j}{d_0 + \ell - 1}. $  This function
 yields the standard algorithm for all linear divisor methods.
 Then $K_j(\ell)$ is strictly decreasing and therefore has a minimal
 solution which can be found by the analog of algorithm {\sc MinimalSolution}
 by selecting the $M$ largest
 elements in a matrix whose entries are $K_j(\ell).$
 Note that $K_j(\ell)  = \frac{2}{H_j(\ell) + 2}$ and therefore we
 can also find a minimal solution by applying algorithm {\sc
   MinimalSolution} directly
 to the matrix $H_j(\ell).$

\end{proof}

\begin{corollary}\label{cor:rhoround1}
Let $\varphi$ be a  symmetric and strictly convex function
with $ \varphi (0)=0$  and $\varphi(x) >0$
for all $ x>0$. Let
$$\psi({\bf m}) = \sum_{j=0}^n \varphi(|m_j-q_j^\rho|)=
 \sum_{j=0}^n \varphi(|m_j-\lfloor q_j^\rho \rfloor -r_j^\rho|).
$$ 
Then the algorithm {\sc  MinimalSolution} with the 
  decomposable error function $\psi$
yields the $\rho$-rounding method.
\end{corollary}

\begin{proof}
Recall  that for  the $\rho$-rounding  method we  let ${\bf  q}^\rho =
(q_1^\rho, \ldots, q_n^\rho)$  by defining $q_j^\rho = \frac{q_j}{M}(M
+2\rho  -1 ).$  Define $r_j^\rho$  by  $ q_j^\rho  = \lfloor  q_j^\rho
\rfloor +  r_j^\rho.$ We shall  see that the function  $\psi({\bf m})$
yields  the $\rho$-rounding  method.   For $j  \in  \{1, \ldots,  n\}$
define as in  Theorem~\ref{the:main} $H_j(x) = \varphi(|x-q_j^\rho|) -
\varphi(|x-q_j^\rho-1|).$  As  $\varphi(x)$  is  strictly  convex,  it
follows that  $H_j(x)$ is strictly increasing  for all $x  \in \R.$ In
particular,
$$
H_j( \lfloor q_j^\rho\rfloor ) \le H_j( q_j^\rho )
 =  -\varphi(1)  < 0.$$
 Also, $$- \varphi(1) =  H_j(q_j^\rho) < H_j( \lfloor q_j^\rho \rfloor
 + 1) \le H_j(\ell)$$ for  all $\ell \ge \lfloor q_j^\rho\rfloor + 1.$
 Hence  the   union  of  the  sets  $L_j=\{\ell   \mid  H_j(\ell)  \le
 -\varphi(1)\}$  contain  $  \sum_{j=1}^n  \lfloor  q_j^\rho  \rfloor$
 elements.  By  Equation (\ref{eq:mm}) these are at  most $M$ elements
 for $0< \rho < 1$  and at most  $M-1$ elements for $\rho=0$  and at
 most   $M+1$  elements   for  $\rho=1.$   Therefore   algorithm  {\sc
 MinimalSolution}  allocates to party  $P_j$ a  number $m_j^0 \ge |L_j|$
 seats. If $\rho  \le 1$, each party obtains  at least $\lfloor q_j^\rho
 \rfloor$ seats and the remaining seats are allocated according to the
 smallest elements among $H_j( \lfloor q_j^\rho \rfloor + 1)$ for $1 \le j
 \le n$.   Since $H_j(\lfloor q_j^\rho \rfloor + 1) =
 \varphi(1-r_j^\rho) - \varphi(r_j^\rho)$ and $\varphi$ is strictly
 convex we have
$ \varphi(1-r_j^\rho) - \varphi(r_j^\rho) \le 
 \varphi(1-r_k^\rho) - \varphi(r_k^\rho)$ if and only if $r_j^\rho \ge
 r_k^\rho.$ Thus the
remaining seats are allocated according to the greatest remainders.
If $\rho=0$  and all $r_j^\rho=0$ the additional seat is allocated at
 random which corresponds to choosing a random minimal solution.
If $\rho=1$  and all $r_j^\rho=0$ we have  assigned $M+1$ seats.
One seat is
 taken from a random  party which corresponds to choosing a random
 minimal solution. This yields precisely the $\rho$-rounding
 method.
\end{proof}

Note that this was proved by P\'olya (1919), \cite{Polyaa}, for $\rho=1/2$.

\begin{corollary}\label{cor:rhoround2}
$$\psi({\bf m}) = \sum_{j=0}^n |m_j-q_j^\rho|=
 \sum_{j=0}^n |m_j-\lfloor q_j^\rho \rfloor -r_j^\rho|.
$$ 
The $\rho$-rounding method  can
 be obtained from the algorithm {\sc  MinimalSolution} for the 
  decomposable error function $\psi$.
\end{corollary}

\begin{proof}
Choose $\varphi_j(x) = |x-q_j|$. Note that this function is
 convex,
though  not
strictly convex. 
Now
$$H_j(x) = |x-q_j^\rho| - | x-q_j^\rho-1| = \left\{ \begin{array}{ll}
-1 & x < q_j^\rho  \\
2x-2q_j^\rho -1 & q_j^\rho \le x \le q_j^\rho + 1 \\
 1 & x > q_j^\rho + 1
 \end{array}
 \right..$$
Observe that $H_j(x)$ is strictly increasing for 
$q_j^\rho \le x \le q_j^\rho + 1.$ Thus the proof of
Corollary~\ref{cor:rhoround1} immediately generalises to this situation.
\end{proof}

\begin{corollary}\label{cor:rhoround3}
Let
$\varphi(x) = (x-q_j^\rho)^p$ for all  $p \ge 1$ and let
$\psi({\bf m}) = \sum_{j = 1}^n | q_j^\rho - m_j |^p$.
Then the algorithm {\sc  MinimalSolution} with the 
  decomposable error function $\psi$
yields the $\rho$-rounding method.
\end{corollary}

Obviously one can also take the  $p$-th root of $\psi({\bf m})$ as the
function $\psi$ and this shows  that all $\ell_p$-norms can be used as
error functions.  In  particular for $p = 2$  the $\ell_2$-norm is the
usual  Euclidean  distance.   Thus  the  $\rho$-rounding  method  also
minimises the Euclidean distance between the points ${\bf q}^\rho$ and
${\bf m}_0,$ and thus yields  the closest integer valued lattice point
for each vector ${\bf q}$ of exact quotas.  
Note that this was known to P\'olya (1918), \cite{Polyab}, for general
convex functions $\varphi$ 
and to Birkhoff (1976), \cite{Birkhoff76}, for the $\ell_p$-norms.
Finally we emphasize that
$$ \lim_{p\rightarrow\infty} \psi_p({\bf m}) = \psi_{\infty}({\bf m}) =
Max_{j=1,\ldots,n}(|m_j - q_j^\rho|) = ||{\bf m} - {\bf q}^\rho||_\infty $$
yields that
the  $\rho$-rounding method also minimises the maximum norm.

 \section{The paradox paradoxa}
 
We finish this paper by discussing certain paradoxa. As all 
linear divisor methods except Sainte-Lagu\"e's and all
$\rho$-rounding methods except Hare's are biased, we
restrict our attention to these two methods. First we show, by
giving some examples, that the Sainte-Lagu\"e method comes with its
own set of paradoxa. It is very susceptible to
minor variations in quotas ({\sc Instability Paradox}).  We
demonstrate that it violates the Majority Condition in a major way and
finally show that the seat distributions for one party can vary
immensely when the votes of other parties change.

We  then  consider  the  Hare  method.  Here  we  address  the  highly
 publicised {\sc  New State}, {\sc Alabama} and  {\sc Increased Votes}
 paradoxa  and  argue why  we  believe  that  this is  no  paradoxical
 behaviour.  Finally, we  show how  a slight  modification of  the Hare
 method fulfils the majority condition.
 
 \subsection{The method of Sainte-Lagu\"e}

\subsubsection{\sc Instability Paradox} 
 
The   method  of   Sainte-Lagu\"e  displays   the   {\sc  Instability
Paradox}, that is  small variations in the exact  quotas can lead to
large  variations  in  the  seat  allocations. In  the  following  two
examples we see in the case  where one party has the absolute majority
of the votes, variations in the votes of the small parties can lead to
significantly different seat allocations for the major party, without
any changes in the votes of the major party. Similar examples are also
discussed by Huntington (1928), \cite[p. 95]{Hunt}.

For example, suppose  5   parties, 4 of which 
are very small,  are competing in an election for $M=68$ seats.
If the  exact quotas  given by the election are:
$$q_1 =  65.91,\,  q_2= 0.53,\,  q_3=0.521,\, q_4 = 0.52,\,  q_5= 0.519,
$$
then the seat allocation according to Sainte-Lagu\"e is:
 $$m_1 =  64,\,      m_2=m_3=m_4=m_5=1 $$
whereas for a slightly different election result for the 5 parties, namely
 $$q_1 =  66.075,\, q_2=0.485,\, q_3=  0.481,\, q_4=0.48 \mbox{  and }
q_5 = 0.479$$ the seat allocation according to Sainte-Lagu\"e is now:
 $$m_1  = 68,\,  m_2 =  m_3=m_4 =  m_5 =  0.$$ A  change of  the exact
proportions by 0.165 seats enforces upon the large party a change of 4
seats. 
With Hare's  method in both cases
the seat allocation would be $m_1=66,\, m_2=m_3=1,\, m_4=m_5=0 $.

Of course  these are  extreme examples, but  even under such  conditions an
election method has to be able  to prove itself.

\subsubsection{\sc The Majority Paradox}

The Sainte-Lagu\"e method also displays the {\sc Majority
  Paradox}.  It can happen that a party obtains more than 50 \% of the
  votes but receives several seats less than 50 \% of the seats. 

For example, at a community election of a town council with $M=51$ seats the
following exact quotas might occur:
$$
\begin{array}{llll}
q_1=26; &  q_2= 7.96;&  q_3= 5.84;& q_4= 4.78;\\
 q_5= 3.72; & q_6= 1.60; & q_7= 0.56;& q_8= 0.54.
\end{array}
$$
Then the seat allocation according to Sainte-Lagu\"e is:
$$m_1= 24;\, m_2= 8;\, m_3= 6;\, m_4= 5;\,  m_5= 4;\, m_6= 2;\, m_7=
m_8 = 1. $$ 

Even though Party 1 won the absolute majority of the votes and was even
allocated 26 seats by the exact quota, it looses 2 seats through the
allocation method and therefore looses the absolute majority in the
council.

\subsubsection{\sc Vote Stability Paradox}
This example also displays the {\sc Vote Stability Paradox}. 
A slightly different election result could have been:
$$
\begin{array}{llll}
q_1= 26;& q_2= 8.03;& q_3= 7.09;& q_4= 6.12;\\
 q_5= 1.415;& q_6=1.405;& q_7= 0.472;& q_8= 0.468.
\end{array}
$$
Then the seat allocation according to Sainte-Lagu\"e is:
$$m_1= 28;\, m_2= 8;\, m_3= 7;\, m_4= 6;\, m_5= m_6 = 1;\, m_7=m_8= 0.$$
Note also a variation of 4 seats for the largest party between these two
elections, even though in both cases the number of votes for the largest
party was the same. The smaller parties dictated the outcome. Apart from
these examples, especially if a 5\% hurdle is installed, the results of
Sainte-Lagu\"e and Hare's procedure are almost always the same.  For many
Federal elections Sainte-Lagu\"e and Hare differed in the seat
allocation for the 
Lower House  (Bundestag) only in two cases. The effect of the 5\% hurdle is
much larger than the change of apportionment method. However, if the 5\%
hurdle is abolished, the effect of the apportionment methods becomes much
larger and is sometimes surprising, as the examples above point out.

Note that this cannot happen in any $\rho$-rounding method as the seat
variations for a given party for constant $M,$ $A$ and number of votes
is at most one seat.

 \subsection{The Hare method}

\subsubsection{\sc New State Paradox}
Many people think there is a paradox in the 
method of the greatest remainder. Especially, the {\sc New State
Paradox} (or ``Parteizuwachsparadox''): if a new party enters the
apportionment method without 
changing any of the original votes of the other parties, then it can
happen that the new party gets a certain number of seats but in
addition, there is a redistribution of the other parties too.  For
instance,
Pukelsheim (1989), \cite[Table~6]{Pukelsheim392}, gives the following example.
Consider parties $A, B, C$ and $D$ which each have
attained $320, 238, 79,$ and $17$ votes, respectively. If we
distribute 37 seats among parties $A, B,$ and $C$, then according to
the method of Hare, they receive $18$, $14$ and $5$ seats, respectively.
However, if we distribute $37+1$ seats among
 parties $A, B,C$, and $D$ then
they receive $19$, $14$, $4$ and $1$ seats, respectively.
This seems to be a paradox as nothing has changed in the votes for the
parties $A$, $B$, and $C$.  Nevertheless, the party $A$ took one seat
away from party $C$.

To understand this behaviour, we have to calculate the exact quotas
$q_A, q_B, q_C$ for parties $A$, $B$ and $C$,
which are 
\begin{eqnarray*}
q_A &= & 320 \cdot 37/ 637 = 18.587127,\\
q_B & = &  238 \cdot 37 / 637 = 13.824175,\\
q_C & = &  79 \cdot 37 / 637 = 4.588697.
\end{eqnarray*}

If we add the votes for party $D$ then we obtain at total of $654$
votes and the exact quotas $q'_A, q'_B, q'_C, q'_D$ for parties
$A$, $B$, $C$ and $D$ are 
\begin{eqnarray*}
q'_A &=&  320 \cdot 38/ 654 = 18.593272, \\
q'_B &=& 238 \cdot 38 / 654 = 13.828746,\\
q'_C &=&  79 \cdot 38 / 654 = 4.590214, \\
q'_D &=&  17 \cdot 38 / 654 =.987767.
\end{eqnarray*}

In our opinion this behaviour does not deserve the label \emph{paradox}
as it is easily explained. The new party $D$ changes the exact quotas
of all other parties and hence to obtain a fair apportionment
of the seats as close to the exact quotas as possible, it is only fair
that the seat allocation changes. The literature which label this behaviour a
\emph{paradox}  avoids the exact quotas like the bubonic plague.

\subsubsection{Alabama Paradox}

The  Alabama  Paradox (Mandatszuwachsparadox)  appeared  first in  the
United States.  If one increases the  number of seats from $M$ to $M'$
then it may happen that a party receives more seats when $M$ seats are
distributed than  when $M'$ are  distributed. When computing  all seat
allocations  for  the  American  House  of  Representatives  using  the
$\rho$-rounding method for $\rho =  1/2$ and the election results from
1880 and  varying the  possible seat numbers  between 275 and  350, the
chief Clark of the Census office noticed that for $M=299$ and $M'=300$
the number of  representatives for Alabama decreased from  8 to 7, see
Balinski and Young (1982), \cite[Table 5.1]{BalinskiYoung}. This
cannot happen in linear divisor methods.
 
Since the Hare method respects quotas it may happen that the exact quotas
change by increasing the number of seats. As a result the number of
seats a party receives might fall back to its lower quota. Perhaps a
good way to explain the situation is to think of the lower quotas as
the guaranteed number of seats for each party and an additional seat
as a bonus. The bonuses are then distributed according to greatest
claim. They are not guaranteed. It can thus happen that with different
number of seats the bonus seats are reallocated. 

Consider for  example three parties who received  the following votes:
$a_1=  107890192;$ $a_2=  197827864;$  and $a_3=  18986361$. Thus  the
total number  of votes is $A =  324704417.$ The following table lists
the
exact quotas and the seat allocations using Hare's and
Sainte-Lagu\"e's method for $M=94$ seats and $M=95$ seats.

$$\begin{array}{l|l|lll}
 & \mbox{Exact Quotas} & q_1=  31.2336 &  q_2 =57.2700 &  q_3=
5.4964 \\
\hline
\hline
M=94 & \mbox{Hare} & m_1= 31 & m_2=57 & m_3= 6 \\
     & \mbox{Sainte-Lagu\"e} & m_1=31&  m_2=57& m_3= 6 \\
\hline
\hline
M = 95& \mbox{Hare} & m_1= 32& m_2=58& m_3= 5 \\
     & \mbox{Sainte-Lagu\"e} & m_1=31&  m_2=58& m_3= 6 \\
  \end{array}$$

While at a superficial glance it seems unfair that the third party
should have a seat less when more seats in total are allocated, a look
at the exact quotas explains the situation.
By the addition of one seat all exact quotas increase by the same
percentage. The $\rho$-rounding method is closer to the ideal of being as
close to exact proportion as possible.

{\bf Acknowledgements}
We thank the editor and an anonymous referee for their valuable
suggestions and  Mary Niemeyer for helpful comments.

\end{document}